\documentclass[12pt]{amsart} 
 
\usepackage[psamsfonts]{amssymb} 
\usepackage{amsfonts,amsmath} 
\usepackage{epsfig,multicol} 

\newtheorem{theorem}{Theorem}[section]

\newtheorem{lemma}[theorem]{Lemma}
\newtheorem{corollary}[theorem] {Corollary}
 
\begin{document} 
 
\title{Indecomposable $PD_3$-complexes}
\author{J.A.Hillman}
\address{School of Mathematics and Statistics\\
     University of Sydney, NSW 2006\\
      Australia }

\email{jonathan.hillman@sydney.edu.au}
 
\thanks{This work was begun at Sydney in 2008 and continued
at the MPIM Bonn and at the Universit\'e Paul Sabatier in Toulouse.
The paper was largely written while the author 
was visiting the University of Durham
as the Grey College Mathematics Fellow for Michaelmas Term of 2008.
The main result was completed in Sydney, in 2011.
}

%\date{6 Jan 2008}
 
\begin{abstract} 
We show that if $X$ is an indecomposable $PD_3$-complex and $\pi_1(X)$
is the fundamental group of a reduced finite graph of finite groups 
but is neither $Z$ nor $Z\oplus{Z/2Z}$ then $X$ is orientable,
the underlying graph is a tree,
all the edge groups are $Z/2Z$ and all but at most one of the vertex groups 
is dihedral of order $2m$ with $m$ odd.
The remaining vertex group has cohomological period dividing 4.
Every such group is realized by some $PD_3$-complex.

We also ask whether every $PD_3$-complex has a finite covering space
which is homotopy equivalent to a closed orientable 3-manifold,
and we propose a strategy for tackling this question.
\end{abstract}


 
\maketitle 

\section{Introduction}

It is a well known consequence of the Sphere Theorem that every closed 3-manifold 
is a connected sum of indecomposable factors, 
which are either aspherical or have fundamental group $Z$ or a finite group.
There is a partial analogue for $PD_3$-complexes: 
Turaev showed that a $PD_3$-complex $X$ whose fundamental group 
is a free product is a connected sum
while Crisp showed that every indecomposable $PD_3$-complex is either aspherical or
its fundamental group is the fundamental group of a finite graph of finite groups.
However the group may have infinitely many ends, 
in contrast to the situation for 3-manifolds. 
Two orientable examples with group $S_3*_{Z/2Z}S_3$ were given in 
\cite{hi04,hi05}.

We shall show that, excepting only the cases $S^1\times{RP^2}$ 
and $S^1\widetilde\times{S^2}$,
every indecomposable $PD_3$-complex with virtually free fundamental group 
is orientable,
the underlying graph is a tree,
all the edge groups are $Z/2Z$ and all but at most one of the vertex groups 
is dihedral of order $2m$ with $m$ odd.
The remaining vertex group has cohomological period dividing 4.
(We may also assume that the tree is linear.)
Every group with such a graph of groups structure is realized 
by some $PD_3$-complex.
(It remains unclear whether the existence of indecomposable examples
with infinitely-ended group is merely an accident of nature
or has some deeper explanation.)

Our argument relies on Turaev's criterion for a group to be the fundamental group
of a $PD_3$-complex, and on one of Crisp's results,
in which he showed that if the centralizer of an element of $\pi=\pi_1(X)$ 
of prime order $p>1$ is infinite then $p=2$ and the element is orientation-reversing.
In conjunction with Turaev's Splitting Theorem it follows quickly that 
(in the orientable case) the Sylow subgroups of the vertex groups
in a graph of groups structure for the fundamental group 
are cyclic or quaternionic.
Hence the vertex groups all have periodic cohomology.
In our main result, Theorem 5.2, 
we use the known classification of such groups 
with Crisp's result to restrict the possible vertex and edge groups.
The constructive aspect is an extension of the idea in \cite{hi04},
in which we showed that the augmentation ideal for $S_3*_{Z/2Z}S_3$
had a self-conjugate, diagonal presentation matrix.
Crisp's result is used again to show that there are 
no exotic nonorientable examples.

In the final part of this paper we turn to the aspherical case.
Here the main question is whether every aspherical $PD_3$-complex 
is homotopy equivalent to a closed 3-manifold.
An equivalent question is whether every $PD_3$-complex 
has a finite covering space
which is homotopy equivalent to a closed orientable 3-manifold.
We suggest a reduction of this question to a question about 
Dehn surgery on links.

\section{group theoretic preliminaries}

If $G$ is a group $|G|$, $G'$ and $\zeta{G}$ shall denote the order, 
commutator subgroup and centre of $G$, while if $H\leq{G}$ is a subgroup $C_G(H)$
and $N_G(H)$ shall denote the centralizer and normalizer, respectively.
Let $I_G$ denote the augmentation ideal of $\mathbb{Z}[G]$.
A homomorphism $w:G\to\{\pm1\}$ defines an anti-involution of $\mathbb{Z}[G]$
by $\bar{g}=w(g)g^{-1}$, for all $g\in{G}$.

If $R$ is a ring two finitely presentable left $R$-modules 
$M$ and $N$ are {\it stably isomorphic\/} if $M_1\oplus{P}\cong{N}\oplus{Q}$ 
for some finitely generated projective $R$-modules $P$ and $Q$.
Let $[M]$ denote the stable isomorphism class of $M$.
If $I_G$ has a finite presentation matrix $A$ over $\mathbb{Z}[G]$
let $J_G$ be the left $\mathbb{Z}[G]$-module 
with presentation matrix the conjugate transpose $\overline{A}^{tr}$.
Tietze move considerations show that $J(G)$ is well-defined up to 
stabilization by direct sums with finitely generated {\it free\/} modules.
In particular, $[J_G]$ is well-defined \cite{tur}.

If all the Sylow subgroups of a finite group $M$ are cyclic then $M$ is metacyclic, 
with a presentation
\[\langle a,b\mid ~a^n=b^m=1,~aba^{-1}=b^r\rangle,\]
where $r^n\equiv1$ {\it mod} $m$ and $(m,n(r-1))=1$, so $m$ is odd.
(See Proposition 10.1.10 of \cite{rob}.)
Let $u=\min\{ k\mid r^k\equiv1~mod~m\}$.
Then $M'$ and $\zeta{M}$ are generated by the images of $b$ and $a^u$, respectively.
When $n=2$ and $r=-1$ we have the dihedral group $D_{2m}$.
If we set $m=2s+1$ then $D_{2m}$ has the presentation
\[\langle a,b\mid~ a^2=1,~ab^sa=b^{s+1}\rangle.\]

There are six families of finite groups with periodic cohomology:
\begin{enumerate}
\item $Z/mZ\rtimes{Z/nZ}$;

\item $Z/mZ\rtimes(Z/nZ\times{Q(2^i)})$, $i\geq3$;

\item $Z/mZ\rtimes(Z/nZ\times{T^*_k})$, $k\geq1$;

\item $Z/mZ\rtimes(Z/nZ\times{O^*_k})$, $k\geq1$;

\item $(Z/mZ\rtimes{Z/nZ})\times{SL(2,p)}$, $p\geq5$ prime;

\item $Z/mZ\rtimes(Z/nZ\times{TL(2,p)})$, $p\geq5$ prime.

\end{enumerate}
Here $m$, $n$ and the order of the quotient by the metacyclic subgroup
$Z/mZ\rtimes{Z/nZ}$ are relatively prime.
(See \cite{dm85}.)
The groups $TL(2,p)$ of the final family may be defined as follows.
Choose a nonsquare $\omega\in\mathbb{F}_p^\times$, and let $TL(2,p)\subset{GL(2,p)}$
be the subset of matrices with determinant 1 or $\omega$.
The multiplication $\star$ is given by $A\star{B}=AB$ if $A$ or $B$ has determinant 1, 
and $A\star{B}=\omega^{-1}AB$ otherwise.
Then $SL(2,p)=TL(2,p)'$ and has index 2.
(Note also that $SL(2,3)\cong{T^*_1}$ and $TL(2,3)\cong{O^*_1}$.)

In particular, a finite group has cohomological period 2 if and only if it is cyclic, 
and has cohomological period 4 if and only if it is a product $B\times{Z/dZ}$,
where $B$ is a generalized quaternionic group $Q(8a,b,c)$, 
an extended binary polyhedral group $T^*_k$, $O^*_k$ or $I^*=SL(2,5)$
or a metacyclic group (with $n=2^e$ and $r=-1$), and $(d,|B|)=1$ \cite{dm85}.

\begin{lemma}
Let $G$ be a finite group with periodic cohomology.
If $G$ is not cyclic or metacyclic then it has an unique central involution
which is a square, and $4$ divides $|G|$.
\end{lemma}

\begin{proof}
This follows on examining the above list of finite groups with periodic cohomology.
Since all subgroups of order $p^2$ in a finite group $G$ with periodic cohomology are cyclic, 
an involution $g\in{G}$ is central if and only if it is the unique involution.
\end{proof}

In particular, if $G$ has cohomological period 4 and does not have a central involution then
$G\cong{D_{2m}}\times{Z/dZ}$, for some odd $m\geq3$ and $d\geq1$.

\begin{lemma}
Let $G$ be a finite group with periodic cohomology of period greater than $4$.
Then $G$ has a subgroup $H\cong{Z/pZ}\rtimes{Z/qZ}$, where $p$ is an odd prime,
$q$ is an odd prime or $4$, $q$ divides $p-1$ and $\zeta{H}=1$.
\end{lemma}

\begin{proof}
This follows on examining the above list of finite groups with periodic cohomology.
\end{proof}

Such groups $H$ have presentations
$\langle a,b\mid ~a^q=b^p=1,~aba^{-1}=b^r\rangle,$
where $r$ is a primitive $q$th root {\it mod\/} $p$.

A {\it graph of groups} $(\mathcal{G},\Gamma)$ consists of a graph $\Gamma$ 
with origin and target functions $o$ and $t$ from the set of edges $E(\Gamma)$ 
to the set of vertices $V(\Gamma)$, and a family $\mathcal{G}$ of groups $G_v$ 
for each vertex $v$ and subgroups $G_e\leq{G_{o(e)}}$ for each edge $e$,
with monomorphisms $\phi_e:G_e\to{G_{t(e)}}$.
(We shall usually suppress the maps $\phi_e$ from our notation.)
In considering paths or circuits in $\Gamma$ we shall not require that 
the edges be compatibly oriented.

The {\it fundamental group} of $(\mathcal{G},\Gamma)$ is the group 
$\pi\mathcal{G}$ 
with presentation
\[\langle G_v,t_e\mid~t_egt_e^{-1}=\phi_e(g)~\forall{g}\in{G_e},~t_e=1~\forall{e}\in{E(T)}\rangle,\]
where $T$ is some maximal tree for $\Gamma$.
Different choices of maximal tree give isomorphic groups.
We may (and shall) always assume that the graph of groups is {\it reduced},
i.e., that if $o(e)\not=t(e)$ then $G_e$ is properly contained in each 
of $G_{o(e)}$ and $G_{t(e)}$.
(See \cite{dd}.)
If there is an edge with $G_e=G_{o(e)}$ and $\phi_e:G_e\cong{G_{t(e)}}$ 
we shall say that the graph of groups has a {\it loop isomorphism}.

\begin{lemma}
Let $\pi=\pi\mathcal{G}$, where $(\mathcal{G},\Gamma)$ 
is a nontrivial reduced finite graph of groups.
If there is an edge $e$ with $G_e=1$ then either $\pi$ 
is a nontrivial free product or $\pi\cong{Z}$.
\end{lemma}

\begin{proof}
If $\Gamma-\{e\}$ has two components then $\pi$ is a nontrivial free product.
Otherwise a maximal tree for $\Gamma-\{e\}$ is also a maximal tree for
$\Gamma$, and the stable letter $t_e$ generates a free factor of $\pi$.
\end{proof}

The following converse is due to Daniel Groves.

\begin{lemma}
Let $\pi=\pi\mathcal{G}$, where $(\mathcal{G},\Gamma)$ is a finite graph of
finite groups and $\Gamma$ is a tree.
If all the edge groups are nontrivial then $\pi$ is indecomposable.
\end{lemma}

\begin{proof}
If $\pi\cong{A*B}$ then $\pi$ acts without global fixed points on the 
Bass-Serre tree $\Upsilon$ associated to the splitting.
Each finite subgroup of $\pi$ fixes a point in this tree.
If $x_o$ and $x_t\in\Upsilon$ are fixed by 
adjacent vertex groups $G_{o(e)}$ and $G_{t(e)}$ 
then $G_e$ fixes the interval $[x_o,x_t]$ joining these points.
Hence $x_o=x_t$, since edge-stabilizers in $\Upsilon$ are trivial.
Induction on the size of $\Gamma$ now shows that 
$x_o$ is fixed by $\pi$, contradicting the fitst sentence of the proof.
\end{proof}

This argument extends easily to all finite graphs of finite groups with non-trivial edge groups \cite{gr11}, but we need only the above case.
 
\begin{lemma}
Let $\pi=\pi\mathcal{G}$, where $(\mathcal{G},\Gamma)$ is a finite graph of groups.
If $C$ is a subgroup of an edge group $G_e$ with $N_{G_e}(C)$ properly contained in each of 
$N_{G_{o(e)}}$ and $N_{G_{t(e)}}$ then $N_\pi(C)$ is infinite.
\end{lemma}

\begin{proof}
If $g_o\in{G_{o(e)}}-G_e$ and $g_t\in{G_{t(e)}}-G_e$ each normalize $C$ then 
$g_og_t$ normalizes $C$ and has infinite order in $\pi$.
\end{proof}

\section{Turaev's criterion and Crisp's Theorem}

If $K$ is an $n$-dimensional complex and $w:\pi=\pi_1(K)\to\{\pm1\}$ 
is a homomorphism let $C_*=C_*(\widetilde{K})$ be the cellular chain complex 
of the universal cover and let $DC_*$ be the dual chain complex with
$DC_q=\overline{Hom_{\mathbb{Z}[\pi]}(C_{n-q},\mathbb{Z}[\pi])}$ given by
dualizing, defining a left module structure by $(g\delta)(c)=w(g)\delta(c)g^{-1}$
for all $g\in\pi$, $\delta\in{DC_q}$ and $c\in{C_{n-q}}$, and reindexing.
Then $K$ {\it satisfies Poincar\'e duality with local coefficients 
and orientation character $w$}
if and only if $H_n(\mathbb{Z}^w\otimes_{\mathbb{Z}[\pi]}C_*)\cong{Z}$ and 
there is a chain homotopy equivalence $DC_*\simeq{C_*}$ 
given by slant product with an $n$-cycle which generates this group \cite{wall}.
We shall call such a complex a $PD_n$-{\it space\/};
it is a $PD_n$-complex if and only if $\pi$ is finitely presentable 
\cite{browd}.
Closed $n$-manifolds are finite $PD_n$-complexes.
Although our main concern in this paper is with $PD_3$-complexes,
we have given the broader definition as $PD_n$-spaces arise naturally 
in connection with Poincar\'e duality groups 
(see \cite{da00} and \S9 below),
and as covering spaces of manifolds \cite{HK}.

In dimensions $n\leq3$ it suffices to know that there there is 
some chain homotopy equivalence $DC_*\simeq{C_*}$.
The next result is substantially based on ideas from \cite{tur},
but has somewhat different hypotheses.
If $M$ is a left $\mathbb{Z}[\pi]$-module 
let $e^jM=Ext^j_{\mathbb{Z}[\pi]}(M,\mathbb{Z}[\pi])$.

\begin{theorem} 
Let $K$ be a connected $3$-complex and $w:\pi=\pi_1(K)\to\{\pm1\}$ 
be a homomorphism.
If $C_*(\widetilde{K})$ is chain homotopy equivalent to a finite projective 
$\mathbb{Z}[\pi]$-complex $C_*$ such that $C_*$ and $DC_*$ 
are chain homotopy equivalent then $K$ is a $PD_3$-space.
\end{theorem}

\begin{proof}
Let ${C_*\otimes_\mathbb{Z}{C_*}}$ have the diagonal left $\pi$-action,
and let $\tau(x\otimes{y})=(-1)^{pq}y\otimes{x}$ for all $x\in{C_p}$ and $y\in{C_q}$.
Let $\Delta:C_*\to{C_*\otimes_\mathbb{Z}{C_*}}$ be an equivariant diagonal.
Then $\tau\Delta$ is also a diagonal homomorphism, and so is chain homotopic to $\Delta$.
Let $\kappa\in{C_3}$ be a 3-chain such that $1\otimes\kappa$ is a cycle representing
a generator $[K]$ of $H_3(\mathbb{Z}^w\otimes_{\mathbb{Z}[\pi]}C_*)\cong
{H_3(\mathbb{Z}^w\otimes_{\mathbb{Z}[\pi]}DC_*)}=H^0(C^*;\mathbb{Z})\cong{Z}$,
and let $\Delta(\kappa)=\Sigma{x_i}\otimes{y_{3-i}}$.
Slant product with $1\otimes\kappa$ defines a chain map $\theta_*:DC_*\to{C_*}$ by
$\theta(\phi)=\Sigma\phi(x_{3-j})y_j$ for all $\phi\in{DC_j}$.
The double dual $DDC_*$ is naturally isomorphic to $C_*$,
and the ``symmetry" of $\Delta$ with respect to the transposition $\tau$ implies that
$D\theta_*$ and $\theta_*$ are chain homotopic, as in \cite{tur}.

Suppose first that $\pi$ is finite.
Then $H^0(C^*)\cong{Z}$ and $H^1(C^*)=0$, so $H_2(C_*)=H_1(C_*)=0$ and $H_3(C_*)\cong{Z}$.
Therefore $\widetilde{K}\simeq{S^3}$ and so $K$ is a $PD_3$-complex by \cite{wall}.

If $\pi$ is infinite $H_3(DC_*)=H^0(C^*)=0$.
Since $H_1(DC^*)=H_1(C_*)=0$ and $H_0(DC_*)=H^3(C^*)\cong{H_0(C_*)}\cong{Z}$,
$H_i(\theta_*)$ is an isomorphism for all $i\not=2$.
In particular, since $H_0(\theta_*)$ is an isomorphism 
the dual $\theta^*:C^*\to{DC^*}$ also induces an isomorphism 
$H^1(C^*)\cong{e^1H_0(C_*)}\cong{H^1(DC^*)}\cong{e^1H_0(DC_*)}$.
Hence $H_2(\theta_*)=H_2(D\theta_*)$ is also an isomorphism,
and so $\theta$ is a chain homotopy equivalence.
Therefore $K$ is a $PD_3$-space.
\end{proof}

A similar (and easier) result is true for complexes of dimension 1 or 2.
On the other hand, the 1-connected space $S^2\vee{S^4}$ is not a $PD_4$-complex,
although it has a cell structure with just 3 cells, 
and its cellular chain complex is obviously isomorphic to its linear dual.

Turaev's characterization of the possible group-pairs $(\pi,w)$ 
of $PD_3$-complexes 
is a fairly straightforward consequence of this theorem.

\smallskip
\noindent{\bf Turaev's Criterion.}
{\sl
Let $\pi$ be a finitely presentable group and $w:\pi\to\{\pm1\}$ a
homomorphism.
Then there is a $PD_3$-complex $K$ with $\pi_1(K)\cong\pi$ and 
$w_1(K)=w$ if and only if $[I_\pi]=[J_\pi]$.
}

\begin{proof}
Let $K$ be a connected $PD_3$-complex with $\pi_1(K)\cong\pi$ and 
$w_1(K)=w$.
We may assume that $K$ has a single $0$-cell and finite 2-skeleton, 
and that $C_*$ and $DC_*$ are finitely generated projective 
$\mathbb{Z}[\pi]$-complexes.
Then $C_0\cong\mathbb{Z}[\pi]$ and 
$\mathrm{Cok}(\partial_2^C)=\mathrm{Im}(\partial_1^C)$ 
is the augmentation ideal $I_\pi$.
The Fox-Lyndon free differential calculus gives a matrix $M$ 
for $\partial_2^C$ with respect to the bases represented 
by chosen lifts of the cells of $K$.
Since $H_0(C_*)\cong{H_0(DC_*)}\cong{Z}$ and
$I_\pi=\mathrm{Cok}(\partial_2^C)$,
Schanuel's Lemma implies that 
$I_\pi\oplus{DC_0}\cong{\mathrm{Cok}(\partial_2^D)}\oplus{C_0}$.
Since $\partial_2^D$ has matrix $\overline{M}^{tr}$ 
it follows that $[I_\pi]=[J_\pi]$.

Conversely, let $K$ be the finite 2-complex associated to a presentation 
for $\pi$, and define $J_\pi$ by means of the Fox-Lyndon matrix.
Suppose first that $J_\pi\oplus\mathbb{Z}[\pi]^m\cong{I_\pi}\oplus\mathbb{Z}[\pi]^n$.
Let $L=K\vee{mD^3}$ be the 3-complex obtained by subdividing the 1-skeleton of $K$
at $n$ points distinct from the basepoint and giving each of the 3-discs 
the cell structure $D^3=e^0\cup{e^2}\cup{e^3}$.
Then $L\simeq{K}$ and $\mathrm{Cok}(\partial_2^L)\cong{I_\pi}\oplus\mathbb{Z}[\pi]^n$.
Let $DC_1=\overline{Hom_{\mathbb{Z}[\pi]}(C_2(\widetilde{L}),\mathbb{Z}[\pi])}$
and let $\alpha:DC_1\to\mathbb{Z}[\pi]$ be the composite
of the projection onto $J_\pi\oplus\mathbb{Z}[\pi]^m$, 
the isomorphism with ${I_\pi}\oplus\mathbb{Z}[\pi]^n$, 
the projection onto $I_\pi$ and the inclusion into $\mathbb{Z}[\pi]$.
Then $\bar\alpha^{tr}:\mathbb{Z}[\pi]\to{C_2}(\widetilde{L})$
has image in $\pi_2(L)=H_2(C_*(\widetilde{L}))$ and so we may attach another $3$-cell along
a map $f$ in the homotopy class of $\bar\alpha^{tr}(1)$.
The resulting $3$-complex $X=L\cup_fe^3$ satisfies the hypothesis of Theorem 3.1, and so
$X$ is a finite $PD_3$-complex with fundamental group $\pi$.
In general, if the projective summands are not stably isomorphic, 
we must adjoin infinitely many 2- and 3-cells, 
to get a finitely dominated $PD_3$-complex,
as in \cite{tur}.
\end{proof}

We should emphasize that this is only part of Turaev's determination of
the characteristic triples $(\pi,w,\mu)$ (with $\mu\in{H_3}(\pi;\mathbb{Z}^w)$) 
realized by $PD_3$-complexes \cite{tur}. (See also \S8 below.)
Turaev also reproved Hendrik's result that the homotopy type of
a $PD_3$-complex is determined by its characteristic triple \cite{He},
and obtained the splitting theorem stated in the first paragraph 
of the introduction as a consequence \cite{tur}.
A similar argument shows that if $\pi$ is $FP_2$ 
(but not finitely presentable) and $[I_\pi]=[J_\pi]$ 
then $(\pi,w)$ is realized by a $PD_3$-space.

We shall use this criterion to exclude some pairs $(\pi,w)$,
usually by means of a homomorphism $f:\mathbb{Z}[\pi]\to{R}$, 
where the ring $R$ is torsion-free as an additive group,
and such that the $\mathbb{Z}$-torsion submodules of $R\otimes_fI_\pi$ 
and $R\otimes_fJ_\pi$ are not isomorphic.
(See Theorems 4.6 and 7.4 below.)
On the other hand, 
we shall justify our constructions of new orientable examples 
by means of Theorem 3.1.

We shall also use repeatedly the following result from 
\cite{crisp} (often together with Lemma 2.5).

\smallskip
\noindent{\bf Crisp's Theorem.}
{\sl
If $X$ is a $PD_3$-complex and $g\in\pi=\pi_1(X)$ has prime order $p$ 
and infinite centralizer $C_\pi(g)$ then $p=2$, 
$g$ is orientation-reversing and $C_\pi(g)$ has two ends.
}

\smallskip
\noindent
Since the automorphism group of a finite group is finite this has the immediate
consequence that if $X$ is orientable and $G$ is a nontrivial finite subgroup
of $\pi$ then $N_\pi(G)$ is finite.

\section{vertex groups have periodic cohomology}

In this section we shall consider orientable $PD_3$-complexes whose 
fundamental groups are fundamental groups of finite graphs of finite groups.
All such groups are finitely presentable.
(However the complexes need not be finite \cite{th}.)

\begin{lemma} 
Let $X$ be an orientable $PD_3$-complex with $\pi=\pi_1(X)\cong\pi\mathcal{G}$,
where $(\mathcal{G},\Gamma)$ is a reduced finite graph of finite groups.
If $(\mathcal{G},\Gamma)$ has a loop isomorphism then $\pi$ 
has a nontrivial free factor.
\end{lemma}

\begin{proof}
If $(\mathcal{G},\Gamma)$ has a loop isomorphism at the edge $e$
then $t_e$ normalizes $G_e$, and so $N_\pi(G_e)$ is infinite.
Therefore $G_e=1$, by Crisp's Theorem, and so $t_e$ generates 
a free factor of $\pi$.
\end{proof}

A finitely generated group is the fundamental group of a finite graph 
of finite groups if and only if it is virtually free.
(See Corollary IV.1.9 of \cite{dd}.)
If $\pi$ has a free normal subgroup $F$ of finite index then the canonical
surjection $s:\pi\to{G}=\pi/F$ is injective on every finite subgroup of $\pi$.
In particular, if $H$ is a finite subgroup of $\pi$ then the subgroup $FH=s^{-1}s(H)$ 
generated by $F$ and $H$ is a semidirect product $F\rtimes{H}$.

\begin{lemma} 
Let $X$ be an indecomposable orientable $PD_3$-complex.
If $\pi=\pi_1(X)$ has a free normal subgroup $F$ such that $\pi/F$ 
is a finite nilpotent group then $\pi$ is cyclic or 
$\pi\cong{Q(2^k)}\times{Z/dZ}$ for some $k\geq3$ and odd $d$.
\end{lemma}

\begin{proof}
If $\pi$ has a free factor then $\pi\cong{Z}$.
Otherwise we may assume that $\pi=\pi\mathcal{G}$,
where $(\mathcal{G},\Gamma)$ is a reduced finite graph of finite groups
with no loop isomorphisms.
Thus each edge group $G_e$ is a proper subgroup of each of $G_{o(e)}$ and $G_{t(e)}$.
The vertex groups are nilpotent since they map injectively to $\pi/F$.
Hence the normalizer of $G_e$ in each of $G_{o(e)}$ and $G_{t(e)}$ is strictly
larger than $G_e$, since nilpotent groups satisfy the normalizer condition.
(See Chapter 5, \S2 of \cite{rob}.)
Hence $N_\pi(G_e)$ is infinite, by Lemma 2.5, and so $G_e=1$.

Since $X$ is indecomposable so is $\pi$,
and since $\pi$ has no free factor $\Gamma$ 
has one vertex and no edges.
Hence $\pi$ is finite, and so $\widetilde{X}\simeq{S^3}$.
Therefore $\pi$ has cohomological period dividing 4.
Since it is nilpotent it is cyclic or the direct product of a cyclic group
of odd order with a quaternionic 2-group $Q(2^k)$, for some $k\geq3$.
\end{proof}

\begin{theorem} 
Let $X$ be an orientable $PD_3$-complex with $\pi=\pi_1(X)\cong\pi\mathcal{G}$,
where $(\mathcal{G},\Gamma)$ is a reduced finite graph of finite groups.
Then the vertex groups have periodic cohomology and the edge groups are metacyclic.
\end{theorem}

\begin{proof}
Let $F$ be a maximal free normal subgroup of $\pi$.
If $S$ is a Sylow $p$-subgroup of a vertex group $G_v$ then $FS$ 
is the fundamental group of a finite graph of $p$-groups.
The indecomposable factors of $FS$ are either infinite cyclic or are
finite and have periodic cohomology, by Lemma 4.2.
Therefore $S$ has periodic cohomology.
Since a finite group has periodic cohomology if and only if this holds 
for all its Sylow subgroups (see Proposition VI.9.3 of \cite{brown})
it follows that $G_v$ has periodic cohomology.

If $G_e$ is not metacyclic it has a central involution, 
which is a square,
by Lemma 2.1.
This involution is orientation preserving, 
and is also central in each of $G_{o(e)}$ and $G_{t(e)}$,
since they cannot be metacyclic.
This contradicts Crisp's Theorem.
\end{proof}

\begin{corollary} 
For any edge $e$ at least one of the vertex groups $G_{o(e)}$ or $G_{t(e)}$ is metacyclic.
If they are each metacyclic then $G_e$ is cyclic. 
\end{corollary}

\begin{proof}
If neither $G_{o(e)}$ nor $G_{t(e)}$ is metacyclic then each has a central involution, $g_o$ and $g_t$, say.
If $|G_e|$ is even then $g_o$ and $g_t$ are each in $\zeta{G_e}$, 
and hence are equal.
But then $N_\pi(g_o)$ contains both vertex groups,
and so is infinite.
If $|G_e|$ is odd it is properly contained in each of its normalizers.
In either case this contradicts Crisp's Theorem.

If $G_{o(e)}$ and $G_{t(e)}$ are each metacyclic then $G_e'$ is normal in each of them,
and so must be trivial, by Crisp's Theorem.
\end{proof}

\begin{corollary}  
If the orders of all the edge groups have a common prime factor $p$
then $\Gamma$ is a tree,
and there is at most one vertex group $V=G_v$ such that $G_e<N_V(G_e)$ for some edge $e$
with $v\in\{o(e),t(e)\}$.
\end{corollary}

\begin{proof}
Let $T$ be a maximal tree in $\Gamma$.
If there is an edge $e$ not in $T$ there is a cycle $\gamma$ in $\Gamma$ incorporating $e$. 
Each vertex group $G_v$ has an unique conjugacy class of subgroups $C_v$ of order $p$,
since its Sylow subgroups are cyclic or quaternionic.
Therefore $t_eC_{o(e)}t_e^{-1}=wC_{o(e)}w^{-1}$, 
where $w$ is a word in the union of the vertex groups along the rest of the cycle.
The element $t_ew^{-1}$ has infinite order, and so $N_\pi(C_{o(e)})$ is infinite.
This contradicts Crisp's Theorem.

If $G_e<N_V(G_e)$ for some $V=G_v$ with $v\in\{o(e),t(e)\}$
we may assume that $C_v\in{G_e}$. 
Then $N_{G_e}(C_v)<N_V(C_v)$, since $C_v$ is unique up to conjugacy in $G_e$. 
Suppose there are two such vertex groups $V=G_v$ and $W=G_w$ with $v\not=w$, 
and choose a (minimal) path connecting these vertices.
As before $C_w=aC_va^{-1}$ for some $a$ in the subgroup generated
by the intermediate vertex groups along the path.
Thus $C_w$ is normalized by the subgroup generated by $N_W(C_w)$ and $aN_V(C_v)a^{-1}$,
which is infinite.
This again contradicts Crisp's Theorem.
\end{proof}

The fact that the Sylow subgroups of a group $G$ have cohomological period
dividing 4 does not imply that $G$ has cohomological period dividing 4.
Nevertheless, this is true in our situation.

\begin{theorem} 
Let $X$ be an orientable $PD_3$-complex with $\pi=\pi_1(X)\cong\pi\mathcal{G}$,
where $(\mathcal{G},\Gamma)$ is a reduced finite graph of finite groups.
Then the vertex groups have cohomological period dividing $4$.
\end{theorem}

\begin{proof}
Let $F$ be a free normal subgroup of finite index in $\pi$.
Suppose there is a vertex group with cohomological period greater than 4. 
Then it has a subgroup $H\cong{Z/pZ}\rtimes{Z/qZ}$ with a presentation 
\[\langle a,b\mid ~a^q=b^p=1,~aba^{-1}=b^r\rangle,\]
where $p$ is an odd prime, $q$ is an odd prime or 4 and 
$r$ is a primitive $q$th root {\it mod\/} $p$.
Let $f:\pi\to\pi/F$ be the canonical projection, and let $FH=f^{-1}f(H)$.
Then $FH\cong F\rtimes{H}$ is the group of an orientable $PD_3$-complex. 
Since every finite subgroup of a free product is conjugate to a subgroup 
of one of the factors we may assume that $\pi=FH$ and is indecomposable.

Assume first that $q$ is an odd prime.
Since $\pi$ is indecomposable and all centralizers of non-identity elements 
are finite we may assume that all edge groups have order $q$.
Since the Sylow $q$-subgroups in each vertex group are all conjugate, 
we may assume also that $\Gamma$ is a tree, by Corollary 4.5,
and that $f$ maps each vertex group isomorphically onto $H$.
It follows that $\pi$ has a presentation 
\[\langle a,b_1,\dots,b_n\mid~a^q=b_i^p=1,~ab_ia^{-1}=b_i^r\rangle.\]
Let $f:\mathbb{Z}[\pi]\to{R}=\mathbb{Z}[Z/qZ]$ be the epimorphism
with kernel the two-sided ideal generated by $\{b_1-1,\dots,b_n-1\}$.
Then $R\otimes_fI_\pi\cong{I_{Z/qZ}}\oplus(R/(p,a-r))^n$.
Hence the $\mathbb{Z}$-torsion of $R\otimes_fI_\pi$ is $(Z/pZ)^n$, with
$a$ acting as multiplication by $r$.

However $R\otimes_fJ_\pi\cong{I_{Z/qZ}}\oplus{N^n}$, 
where $N\cong{R^2/R(p,a^{-1}-r)}$.
Let $\rho=\Sigma_{i<q}a^ir^i$ in $R$. Then 
\[(a^{-1}-r)\rho=a^{-1}(1-a^qr^q)=a^{-1}(1-r^q)\equiv0~{\mathit mod}~p.\]
Therefore $(a^{-1}-r)\rho=p\sigma$ for some $\sigma\in{R}$.
Let $[\rho,\sigma]$ be the image of $(\rho,\sigma)$ in $N$.
Then $[\rho,\sigma]\not=0$, since $p$ does not divide $\rho$ in $R$.
On the other hand $p[\rho,\sigma]=\rho[p,a^{-1}-r]=0$ and
$(a^{-1}-r)[\rho,\sigma]=\sigma[p,a^{-1}-r]=0$.
Thus $a$ acts as multiplication by $r^{-1}$ on 
this nontrivial $p$-torsion element of $N$.
Since $r^{-1}\not\equiv{r}~{\mathit mod}~p$ 
it follows that $R\otimes_fI_\pi$ and $R\otimes_fJ_\pi$ 
are not stably isomorphic, and so $[I_\pi]\not=[J_\pi]$.

If $q=4$ the edge groups have order 2 or 4,
and at least one vertex group has an element of order 4.
We may again assume that $\Gamma$ is a tree, and $\pi$ now has a presentation of the form
\[\langle a,b_1,\dots,b_n\mid~a^4=b_i^p=1,~ab_ia^{-1}=b_i^r~\forall{i}\leq{k},~a^2b_ia^2=b_i^{-1}~\forall{i}>k\rangle,\]
for some $k>1$.
We now find that $a$ acts as multiplication by $r$ on a summand $(Z/pZ)^k$ of 
the $\mathbb{Z}$-torsion of $R\otimes_fI_\pi$,
whereas it acts by $r^{-1}=-r$ on part of the corresponding summand 
of the $\mathbb{Z}$-torsion of $R\otimes_fJ_\pi$.
Therefore we again find that $[I_\pi]\not=[J_\pi]$.

Thus $\pi$ does not satisfy Turaev's criterion.
Hence all vertex groups must have cohomological period dividing $4$.
\end{proof}

It is of course clear that we cannot have $\pi\cong{H}$,
since $H$ has cohomological period $>4$.

\section{the main result}

We shall now use the classification of groups of cohomological period 4 to 
restrict further the possible fundamental groups.

\begin{lemma}
Let $G$ be a finite group with cohomological period $4$, 
and let $C$ be a cyclic subgroup of odd prime order $p$. 
Then $N_G(C)$ is nonabelian unless $p=3$ and $G=B\times{Z/dZ}$ 
with $B=T_1^*$ or $I^*$.
\end{lemma}

\begin{proof}
This follows on examining the list of such groups $G$.
(Note that if $p>5$ then $C$ is central, 
while if $p=5$ and $G=I^*$ or $p=3$ and $G=O^*_1$ then $N_G(C)$ is nonabelian.
If $p=3$ and $G=T^*_k$ or $O^*_k$ with $k>1$ then $C$ is normal in $G$.)
\end{proof}

Let $K_{1,n}$ be the bipartite graph which is topologically a cone 
on $n$ points, with one central vertex $c$, $n$ root vertices
$\{v(i)\mid1\leq{i}\leq{n}\}$
and $n$ oriented edges $\{e(i)=(c,v(i))\mid1\leq{i}\leq{n}\}$.
If $d\geq1$ we may label the root vertices and edges of $K_{1,dn}$ 
as $\{v(i,j)\mid 1\leq{i}\leq{n},1\leq{j}\leq{d}\}$
and $\{e(i,j)\mid 1\leq{i}\leq{n},1\leq{j}\leq{d}\}$,
respectively.

Let $\mathcal{G}$ be a graph of groups with underlying graph $K_{1,n}$,
and such that the vertex group $G_c$ has a subgroup $H$ of finite index $d=[G_c:H]>1$ which contains all the edge groups $G_{e(i)}$.
Then $\pi\mathcal{G}$ has a subgroup $\pi\widetilde{\mathcal{G}}$
of index $d$ where $\widetilde{\mathcal{G}}$
is the graph of groups with underlying graph $K_{1,dn}$
and groups $G_c=H$, $G_{v(i,j)}\cong{G_{v_i}}$ 
and $G_{e(i,j)}\cong{G_{e(i)}}$ for all $(i,j)$.

\begin{theorem}
Let $X$ be an indecomposable orientable $PD_3$-complex with $\pi=\pi_1(X)\cong\pi\mathcal{G}$,
where $(\mathcal{G},\Gamma)$ is a reduced finite graph of finite groups
of cohomological period $4$.
Then $\Gamma$ is a tree, all edge groups are $Z/2Z$,
and at most one vertex group is not dihedral.
\end{theorem}

\begin{proof}
Let $G_e$ be an edge group.
Then $G_e$ is metacyclic, by Theorem 4.3.
If $G_e$ has a central involution then it is also central in
$V=G_{o(e)}$ and $W=G_{t(e)}$, by Lemma 2.1.
This contradicts Crisp's Theorem, and so 4 cannot divide $|G_e|$.

At least one of $V,W$ is metacyclic, by Corollary 4.4.
Suppose that both are metacyclic.
If $C\leq{G_e}$ has odd prime order then $N_V(C)=V$ and $N_W(C)=W$, 
since $V$ and $W$ are metacyclic with cohomological period dividing 4.
As this contradicts Crisp's Theorem $G_e=Z/2Z$.

If $V$ is not metacyclic then it has a central involution, $g$ say, and
$W\cong{D_{2m}}\times{Z/dZ}$ for some relatively prime odd $m\geq3$ and $d\geq1$.
Therefore if $C\leq{G_e}$ has odd prime order $N_W(C)=W$. Hence
$N_V(C)\leq {G_e}$ and so the central involution is in $G_e$.
Moreover, $C_W(g)=G_e$ and so $G_e\cong{Z/2dZ}$.
Since the odd-order subgroup of $G_e$ is central in $W$ its normalizer
in $V$ must be abelian unless $d=3$ or 1, by Lemma 5.3.

Since the edge groups all have even order and groups of cohomological period 4 
and order divisible by 4 have central involutions there is at most one 
such vertex group and $\Gamma$ is a tree, by Corollary 4.5.

If there is an edge $e$ with $|G_e|>2$ then $G_e\cong{Z/6Z}$, the adjacent vertex groups are
${D_{2m}}\times{Z/3Z}$ and $B\times{Z/dZ}$, with $(m,6)=1$, $B=T^*_1$ or $I^*$ and $(d,|B|)=1$,
and the remaining vertex groups are dihedral.
Since $T^*$ and $I^*$ each have a unique element of order 2,
the images of the edge groups in $B\times{Z/dZ}$ all
lie in the same cyclic subgroup of order 6.
Hence $\pi$ has a subgroup of index $|B|/6$ 
with an induced graph-of groups structure with one vertex group $Z/6dZ$, 
$|B|/6$ vertex groups isomorphic to ${D_{2m}}\times{Z/3Z}$.
$|B|/6$ edge groups of order 6,
and all other edge groups of order 2.
(See the two paragraphs preceding this theorem.)
This subgroup is the group of an orientable $PD_3$-complex,
since it has finite index in $\pi$,
and is indecomposable, by Lemma 2.4.
Since $|B|/6\geq4$ this is not consistent with the earlier part of this theorem.
Therefore all edge groups must have order 2.
\end{proof}

Since all involutions in $\pi$ are conjugate we may modify the underlying
graph of groups so that $\Gamma$ is linear: all vertices have valence $\leq2$.

\begin{corollary} 
If all the vertex groups are dihedral then $\pi\cong\pi'\rtimes{Z/2Z}$
and $\pi'$ is a free product of cyclic groups of odd order. 
\qed   
\end{corollary}

Theorem 5.2 and Milnor's theorem on involutions in finite groups 
acting freely on mod-$(2)$ homology spheres together imply 
(without using the Sphere Theorem) that
if $M$ is a closed 3-manifold and $\pi=\pi_1(M)$ is 
freely indecomposable then $\pi$ is finite,
$Z$ or $Z\oplus{Z/2Z}$ or is a $PD_3$-group.
For otherwise $\pi$ would have a finite index subgroup  
$\nu\cong(*_{i\leq{r}}Z/m_iZ)\rtimes{Z/2Z}$, with $m_i$ odd for $i\leq{r}$,
by Theorem 5.2.
Such a group $\nu$ maps onto $D_{2m_1}$ with kernel $\kappa$
a free product of finite cyclic groups of odd order.
Thus $D_{2m_1}$ would act freely on the covering space $M_\kappa$ 
associated to $\kappa$,
which is a mod-$(2)$ homology 3-sphere.
This is impossible, by Milnor's theorem \cite{mil}.

\section{construction}

The Fox-Lyndon presentation matrix for the augmentation ideal of $D_{2m}$
derived from the presentation in \S2 is
$\left(\begin{smallmatrix} a+1 & 0\\
1+ab^s & a\nu_s-\nu_{s+1}
\end{smallmatrix}\right),$
where $\nu_k=1+b+\dots+b^{k-1}$.
The off-diagonal element may be removed by right multiplication by 
$\left(\begin{smallmatrix} 1 & 0\\
1+ab^s & 1
\end{smallmatrix}\right),$
since $(1+ab^s)+(a\nu_s-\nu_{s+1})(1+ab^s)=0$.
On multiplying the second column by $b^{s^2}$ the entries become self-conjugate.

Let $\{G_i\mid 0\leq{i}\leq{n}\}$ be a family of finite groups, 
with $G_0$ having even order and cohomological period $2$ or $4$, 
and $G_i=D_{2m_i}$ being dihedral, with $m_i=2s_i+1$, for $i\geq1$.
Each of these groups has an unique conjugacy class of involutions, 
and so there is a well-defined iterated generalized free product with amalgamation
\[\pi=G_0*_{Z/2Z}G_1*_{Z/2Z}\dots*_{Z/2Z}G_n.\]
We may choose a presentation for $G_0$ with $g$ generators and $g$ relators,
in which the last generator, $a$ say, is an involution.
Taking 2-generator presentations for the dihedral groups, as above, and identifying the involutions,
we obtain a presentation for $\pi$ of the form
\[
\langle G_0, b_1,\dots,b_n\mid~ab_1^{s_1}ab_1^{-1-s_1}=\dots=ab_1^{s_n}ab_1^{-1-s_n}=1
\rangle.
\]
(In particular, such a group has a balanced presentation, with equally many generators and relations.)
The Fox-Lyndon presentation matrix for $I_\pi$ derived from this begins with  a $g\times{g}$
block corresponding to the presentation matrix for $I_{G_0}$ and $n$ new rows and columns.
The elements in the $g$th column and final $n$ rows may be removed and 
the diagonal elements rendered self-conjugate, as before, as the new generators interact only with $a$.
(Note that if $e_1,\dots,e_{g+n}$ are the generators for $I_\pi$ associated to this presentation then
$(a+1)e_g=0$ is a consequence of the first $g$ relations.)

It is now clear that $[I_\pi]=[J_\pi]$, and so $\pi$ is the fundamental group of a $PD_3$-complex.
If $I_{G_0}$ has a square presentation matrix which is conjugate to its transpose 
the argument of \cite{hi04} extends to give an explicit complex with one 0-cell, 
$g+n$ 1-cells, $g+n$ 2-cells and one 3-cell realizing this group.
That this complex is a $PD_3$-complex follows from Theorem 3.1.

The first such group considered in this context was $S_3*_{Z/2Z}S_3$
\cite{hi93,hi04,hi05}, 
but the simplest such example is perhaps $S_3*_{Z/2Z}Z/4Z$, 
with presentation \[\langle a,b\mid~a^4=1,~a^2ba^2=b^2\rangle.\]
This group is realized by a $PD_3$-complex with just six cells.
(In \cite{hi05} we erroneously dismissed this as a possibility.)

\section{indecomposable nonorientable $PD_3$-complexes}

Here we shall show that the only indecomposable nonorientable $PD_3$-complexes 
with virtually free fundamental group are the two $3$-manifolds
${S^1}\widetilde\times{S^2}$ and ${S^1}\times{RP^2}$.

\begin{theorem} 
Let $X$ be an indecomposable nonorientable $PD_3$-complex 
with $\pi=\pi_1(X)\cong\pi\mathcal{G}$,
where $(\mathcal{G},\Gamma)$ is a finite graph of finite groups.
If all the vertex groups are orientation preserving 
then $X\simeq{S^1}\widetilde\times{S^2}$.
\end{theorem}

\begin{proof}
Since $X$ is nonorientable $\pi$ is infinite, 
and is not generated by the vertex groups.
Thus $\Gamma$ is not a tree.
If there were a nontrivial vertex group it would have finite cohomological 
period, and all edge groups would have (orientation preserving) involutions.
But all involutions are conjugate, so $\Gamma$ would be a tree, 
by the argument of Corollary 4.5.
Thus $\pi$ must be a free group.
Since it is infinite and indecomposable it must be $Z$.
The result now follows from \cite{wall}.
\end{proof}

\begin{lemma} 
Let $\pi$ be a finitely presentable group and let 
$f:\mathbb{Z}[\pi]\to{R}=\mathbb{Z}[Z/2Z]=\mathbb{Z}[a]/(a^2-1)$
be the epimorphism induced by an epimorphism $w:\pi\to{Z/2Z}$.
Suppose that $R\otimes_fI_\pi\cong{R/(a+1)}\oplus{T}$,
where $T$ is a $\mathbb{Z}$-torsion module.
Then $[I_\pi]\not=[J_\pi]$.
\end{lemma}

\begin{proof}
Every finitely generated $\mathbb{Z}$-torsion-free 
$R$-module is a direct sum of copies of $R$,
$\mathbb{Z}=R/(a-1)$ and $\mathbb{Z}^w=R/(a+1)$,
and the number of summands of each type is uniquely determined.
(See Theorem 74.3 of \cite{CR}.)
In particular, all finitely generated projective $R$-modules are free,
and so the numbers of summands of types
$\mathbb{Z}$ and $\mathbb{Z}^w$ are invariant under stabilization.

Let $P$ be a presentation matrix for $T$.
Then $A=\left(\begin{smallmatrix}
a+1&0\\ 0& P
\end{smallmatrix}\right)$ is a presentation matrix for  
$R\otimes_fI_\pi$.
The stable isomorphism class $[R\otimes_fJ_\pi]$
contains the module presentated by
$\overline{A}^{tr}=\left(\begin{smallmatrix}
1-a&0\\ 0& \overline{P}^{tr}
\end{smallmatrix}\right)$.
This has $\mathbb{Z}$ as a direct summand,
whereas $R\otimes_fI_\pi$ does not.
Therefore $[I_\pi]\not=[J_\pi]$.
\end{proof}

\begin{lemma} 
Let $X$ be an indecomposable $PD_3$-complex such that $\pi=\pi_1(X)\cong
{F(r)}\rtimes{G}$.
If $\pi$ has an orientation reversing element $g$ of finite order then
$G$ has order $2m$, for some odd $m$.
\end{lemma}

\begin{proof}
If an orientation-reversing element $g$ has order $2^kd$ with $d$ odd 
then $k\geq1$ and $g^d$ is orientation-reversing and of order $2^k$.
Suppose that $|G|$ is a multiple of 4.
We may assume that $G$ is a $2$-group, $\pi$ is indecomposable 
and the graph of groups is reduced.
Then the edge groups must be generated by orientation reversing involutions
and the vertex groups must have order 4, 
by the normalizer condition and Crisp's Theorem.
Since the inclusion of an edge group splits $w$, 
the vertex groups must be $V=(Z/2Z)^2$.
(Thus $k=1$ and each vertex group has two conjugacy classes of
orientation reversing involutions.)

All vertices of the graph $\Gamma$ must have valency at most 2, 
for otherwise there would be an orientation reversing involution 
with centralizer containing $(Z/2Z)*(Z/2Z)*(Z/2Z)$.
Thus either $\Gamma$ is a tree or $\beta_1(\Gamma)=1$.

Let $w=w_1(X)$ and let $f:\mathbb{Z}[\pi]\to{R}
=\mathbb{Z}[Z/2Z]=\mathbb{Z}[a]/(a^2-1)$
be the epimorphism induced by $w$.
Then $f$ induces an epimorphism from
$I_\pi$ to ${I_{Z/2Z}}=R/(a+1)$,
which factors through an epimorphism $h:R\otimes_fI_\pi\to{R/(a+1)}$.
The inclusion of an edge group splits $h$,
and so $R\otimes_fI_\pi\cong{R/(a+1)}\oplus{N}$,
where $N=\mathrm{Ker}(h)$.

If $\Gamma$ is a tree then $\pi$ has a presentation
\[\langle a_1,\dots,a_n,b_1,\dots,b_n\mid~a_i^2=b_i^2=a_ib_ia_i^{-1}b_i^{-1}=1
~\forall~i\leq{n},\]
\[ a_i=a_{i+1}b_{i+1}~\forall~2\leq{i}\leq{n}\rangle,\]
where $w(a_i)=-1$ and $w(b_i)=1$ for all $i\leq{n}$.
(The amalgamations must be essentially as in the final set 
of relations since the edge groups are generated by 
orientation reversing involutions and 
each of the edge group centralizers has two ends.)
In this case consideration of the Fox-Lyndon presentation matrix 
for $R\otimes_fI_\pi$ shows that $\mathbb{Q}\otimes_\mathbb{Z}N=0$.
Thus $N$ is a $\mathbb{Z}$-torsion module, 
so $[I_\pi]\not=[J_\pi]$, by Lemma 7.2.
Therefore $\Gamma$ cannot be a tree.

If $\beta_1(\Gamma)=1$ then $\pi$ has a presentation
\[\langle a_1,b_1,\dots,a_n,b_n,t\mid~
a_i^2=b_i^2=a_ib_ia_i^{-1}b_i^{-1}=1~\forall~i\leq{n},\]
\[ a_i=a_{i+1}b_{i+1}~\forall~2\leq{i}\leq{n},~ta_n=a_1b_1t\rangle,\]
where $w(a_i)=-1$ and $w(b_i)=1$ for all $i\leq{n}$.
After replacing $t$ by $ta_n$, if necessary,
we may assume that $w(t)=1$.
In this case $N=\mathrm{Ker}(h)$ is not a $\mathbb{Z}$-torsion module.
Instead we find that
\[R\otimes_fI_\pi\cong{R/(a+1)}\oplus{(R/(2,a-1))^{n-1}\oplus{M}},
\]
where $M$ is an indecomposable $R$-module with underlying abelian group
$Z\oplus{Z/2Z}$ and $R$-action determined by
$a.(m,[n])=(m,[m+n])$ for all $(m,[n])\in{Z}\oplus{Z/2Z}$.
In particular, the augmentation module $\mathbb{Z}$ is not a summand
of $R\otimes_f{I_\pi}$.
On the other hand, $R\otimes_f{J_\pi}$ does have $\mathbb{Z}$ as a summand.
Therefore $R\otimes_f{I_\pi}$ and $R\otimes_f{J_\pi}$ 
are not stably isomorphic, and so $[I_\pi]\not=[J_\pi]$.

Thus $|G|$ cannot be divisible by 4, and so $|G|=2m$ for some odd $m$.
\end{proof}

In particular, if $w(G_v)\not=1$ then $G_v\cong{Z/mZ}\rtimes{Z/2Z}$ 
for some odd $m$.

\begin{theorem} 
Let $X$ be an indecomposable nonorientable $PD_3$-complex such that 
$\pi=\pi_1(X)$ has an orientation reversing involution.
Then $X\simeq{S^1}\times{RP^2}$.
\end{theorem}

\begin{proof}
Since $\pi$ is indecomposable and has nontrivial torsion $\pi=\pi\mathcal{G}$,
where $(\mathcal{G},\Gamma)$ is a reduced finite graph of finite groups.
At least one vertex group has an orientation reversing element, 
by Theorem 7.1.
If there is an edge $e$ such that $G_{o(e)}$ is orientable and 
$G_{t(e)}$ is nonorientable then $G_e$ must be cyclic of odd order, 
since $G_{t(e)}\cong{Z/mZ}\rtimes{Z/2Z}$ with $m$ odd,
by Lemma 7.3.
But then it is properly contained in each of its normalizers, 
contradicting Crisp's Theorem.
Thus we may assume that all vertex groups are orientation reversing.
Hence they are all such semidirect products, 
and the edge groups are $Z/2Z$.
In particular, each vertex group has an unique conjugacy class
of involutions.

Suppose that there is a vertex group of order $2m>2$.
On passing to a subgroup of finite index, if necessary,
we may assume that $\pi\cong{F(r)}\rtimes{G}$, where $G$ has order $2p$, 
for some odd prime $p$.
Then the vertex groups must all be isomorphic to $G$,
and $G\cong{Z/2p}$ or $D_{2p}$.

Let $T$ be a maximal tree in $\Gamma$.
Then $T$ omits at most one edge of $\Gamma$,
since the centralizer of an involution is finite or has two ends.

Suppose first that $\Gamma$ is a tree.
Let $f:\mathbb{Z}[\pi]\to{R}=\mathbb{Z}[a]/(a^2-1)$
be the epimorphism induced by $w$.
Then $R\otimes_fI_\pi\cong{R/(a+1)}\oplus{M}$, where $M$ is a
$\mathbb{Z}$-torsion module.
Therefore $[I_\pi]\not=[J_\pi]$, by Lemma 7.2,
and so $\Gamma$ cannot be a tree.

If $\beta_1(\Gamma)=1$ then $\pi$ has a presentation
\[\langle a,b_1,\dots,b_n,t\mid~
b_i^p=a_ib_ia_i^{-1}b_i^{-\varepsilon}=a^2=1~\forall~i\leq{n},~ta=at\rangle,\]
where $\varepsilon=1$ if $G$ is cyclic and $\varepsilon=-1$ if $G$ is dihedral.
Moreover, $w(a)=-1$, $w(b_i)=1$ for all $i\leq{n}$ and $w(t)=1$.
Hence 
\[R\otimes_fI_\pi\cong{R/(a+1)}\oplus{R/(a-1)}\oplus(R/(p,a-\varepsilon))^n,\]
and so the $\mathbb{Z}$-torsion of $R\otimes_fI_\pi$ is $(Z/pZ)^n$, 
with $a$ acting as multiplication by $\varepsilon$.
On the other hand,
\[R\otimes_fJ_\pi\cong{R/(a-1)}\oplus{R/(a+1)}\oplus{N^n},\]
where $N\cong{R^2}/R(p,-a-\varepsilon)$ is generated by two elements
$n,n'$, with $pn=(a+\varepsilon)n'$.
Let $\nu=(a-\varepsilon)n$. Then $\nu\not=0$, 
but $p\nu=(a-\varepsilon)(a+\varepsilon)n'=0$
and $(a+\varepsilon)\nu=(a+\varepsilon)(a-\varepsilon)n=0$.
Thus $a$ acts as multiplication by $-\varepsilon$ on 
this nontrivial $p$-torsion element of $N$.
Since $-\varepsilon\not\equiv{\varepsilon}~{\mathit mod}~p$ 
it follows that $R\otimes_fI_\pi$ and $R\otimes_fJ_\pi$ 
are not stably isomorphic, and so $[I_\pi]\not=[J_\pi]$.

Since $\pi$ must be infinite,
the only remaining possibility is that the graph has one vertex 
$v$ and one edge $e$,
with $G_e=G_v=Z/2Z$.
Thus $\pi\cong{Z}\oplus{Z/2Z}=\pi_1(S^1\times{RP^2})$,
and so $X\simeq{S^1}\times{RP^2}$, by \cite{wall}.
\end{proof}

The following corollary strengthens part of Crisp's Theorem.

\begin{corollary}  
Let $X$ be a $PD_3$-complex and $g\in\pi=\pi_1(X)$ a nontrivial element of finite order.
If $C_\pi(g)$ is infinite then $g$ is an orientation-reversing involution 
and $C_\pi(g)=\langle{g}\rangle\times{Z}$.
\qed   
\end{corollary}

\section{homotopy types}

Let $W$ be a $PD_3$-complex with fundamental group $\pi$,
orientation character $w$ and fundamental class $[W]\in{H_3(W;\mathbb{Z}^w)}$.
If $c_W:W\to{K(\pi,1)}$ is a classifying map let $\mu(W)=c_{W*}[W]\in{H_3(\pi;\mathbb{Z}^w)}$.
Two such $PD_3$-complexes $W_1$ and $W_2$ are homotopy equivalent if and only if
$\mu(W_1)$ and $\mu(W_2)$ agree up to sign and the action of $Out(\pi)$ 
\cite{he77}.
If $\pi$ is virtually free then $H_3(W;\mathbb{Z}^w)$ is finite.
Since every indecomposable $PD_3$-complex is either aspherical or has
virtually free fundamental group it follows that there are only finitely many
homotopy types with any given group.
Note also that if $\pi$ is indecomposable and virtually free 
then $Out(\pi)$ is finite \cite{car}, 
and so the group of self-homotopy equivalences of $W$ is finite \cite{He}.

Suppose that $\pi=G_0*_{Z/2Z}\rho$, where $G_0$ has cohomological period dividing 4
and a central involution and $\rho$ is an iterated free product 
of dihedral groups $G_i=D_{2m_i}$ with amalgamation over copies of $Z/2Z$,
where $m_i=2s_i+1$, for $i\leq{n}$.
Then $\rho'\cong*_{i=1}^nZ/m_iZ$.
Let $m_0=|G_0|$.
(We allow the possibility $G_0=Z/2Z$.)
By the work of \S7 above, we may assume that $W$ is orientable.
Since $\rho\cong\rho'\rtimes{Z/2Z}$ we have 
\[H_3(\rho;\mathbb{Z})\cong{H_3(Z/2Z;\mathbb{Z})}\oplus{H_3(\rho';\mathbb{Z})}.\]
A Mayer-Vietoris argument then gives
\[H_3(\pi;\mathbb{Z})\cong{H_3(G_0;\mathbb{Z})}\oplus{H_3(\rho';\mathbb{Z})}
=\oplus_{i=0}^n(Z/m_iZ).\]
Let $f:\pi\to{G_0}$ be the epimorphism with kernel normally generated by $\rho'$, 
and let $W_\sigma$ be the covering space corresponding to $\sigma=f^{-1}(S)$,
where $S<G_0$ is a Sylow $p$-subgroup of $G_0$.
If $p$ is odd $W_\sigma$ is a connected sum of lens spaces, by Theorem 1 of \cite{tur}.
Since $\mu(W_\sigma)$ is the image of $\mu(W)$ under transfer,
it follows that $\mu(W)$ must project to a generator of each 
of the odd cyclic summands of $H_3(\pi;\mathbb{Z})$.
If $p=2$ we may argue instead that the square 
$Sq^1:H^1(W_\sigma;\mathbb{F}_2)\to{H^2(W_\sigma;\mathbb{F}_2)}$ is nonzero.
Hence the generator of $H^3(W_\sigma;\mathbb{F}_2)$ is a product
of elements in the image of $H^1(\sigma;\mathbb{F}_2)$,
by Poincar\'e duality.
It follows that the image of $\mu(W)$ in the 2-primary summand must generate also.

For each $1\leq{i}\leq{n}$ and $u\in{Z/m_iZ^\times}$ there is an automorphism
which sends $b_i$ to $b_i^u$, for $b_i\in{G_i'}$, and which fixes the other vertex groups.
If $G_i\cong{G_j}$ there is an automorphism interchanging $G_i$ and $G_j$.
As every automorphism of $G_0$ fixes the central involution 
it extends to an automorphism of $\pi$ which fixes $\rho$. 
These automorphisms act naturally on $H_3(\pi;\mathbb{Z})$.

In particular, if $G_0=Z/2Z$, so $\pi\cong\pi'\rtimes{Z/2Z}$,
the double cover $W'$ is a connected sum of lens spaces.
Taking into account the actions of these automorphisms
and the homotopy classification of lens spaces, 
we see that $W_1\simeq{W_2}$ if and only if $W_1'\simeq{W_2'}$.

Turaev constructed an isomorphism $\nu$ from $H_3(\pi;\mathbb{Z}^w)$ to a group 
$[F^2(C), I_\pi]$ of projective homotopy classes of module homomorphisms 
and showed that $\mu\in{H_3(\pi;\mathbb{Z}^w)}$ is the image of the fundamental class 
of a $PD_3$-complex if and only if $\nu(\mu)$ is
the class of a homotopy equivalence \cite{tur}.
Since there is at least one homotopy equivalence the ring $End_\pi([I_\pi])$ 
is isomorphic as an abelian group to $\oplus_{i=0}^n(Z/m_iZ)$.
Do the $(n+1)$-tuples of the form $(u_0,\dots,u_n)$ with $(u_i,m_i)=1$ 
for $0\leq{i}\leq{n}$ correspond to the units $Aut_\pi([I_\pi])$?
(This is so when the $m_i$ are all relatively prime,
for then $End_\pi([I_\pi])\cong{Z/\Pi{m_i}Z}$, and
so must act in the obvious way on $H_3(\pi;\mathbb{Z})$.)

We may also ask whether such $PD_3$-complexes can arise in some natural manifold context.
For instance, is $W\times{S^1}$ homotopy equivalent to a closed 4-manifold?
(Since the group of self-homotopy equivalences of such a complex is finite
it is equivalent to ask whether there is a closed 4-manifold $M$ 
with $\chi(M)=0$ and $\pi_1(M)\cong\pi\rtimes{Z}$, 
by Theorem 4.7 of \cite{Hi}.)
The case when $\pi=D_{2m}$ may be ruled out by
a surgery semicharacteristic argument \cite{hm86}.

\section{is every $PD_3$-complex virtually a 3-manifold?}

It is well known that every $PD_2$-complex is homotopy equivalent 
to a closed surface.
The argument of Eckmann and M\"uller \cite{em80} for the cases 
with $\beta_1\not=0$ involves delicate combinatorial group theory.
(The hypothesis $\beta_1\not=0$ is removed in \cite{el81}.)
More recently, Bowditch used geometric group theory to obtain the
stronger result that an $FP_2$ group $\Gamma$ with $H^2(\Gamma;\mathbb{Z}[\Gamma])\cong{Z}$
acts properly discontinuously on $\mathbb{E}^2$ or $\mathbb{H}^2$ \cite{bow}.

Higher dimensional considerations suggest another, 
more topological strategy, which can be justified {\it a posteriori}.
The bordism Hurewicz homomorphism from $\Omega_n(X)$ to 
$H_n(X;\mathbb{Z})$ is an epimorphism in degrees $n\leq4$.
Therefore if $X$ is an orientable $PD_n$-complex with $n\leq4$ 
there is a degree-1 map $f:M\to{X}$ with domain a closed orientable $n$-manifold.
(See \cite{hv} for the corresponding result for nonorientable $PD_n$-complexes, 
using $w_1$-twisted bordism and homology.)
Choose compatible basepoints $m_o$ and $x_o=f(m_o)$,
and let $\pi=\pi_1(X,x_o)$ and $f_*=\pi_1(f)$. 
If $X$ is a finite $PD_2$-complex then such a map 
$f$ is a homotopy equivalence $\Leftrightarrow$ $\mathrm{Ker}(f_*)=1$
$\Leftrightarrow$ $\chi(M)=\chi(X)$.
If $\mathrm{Ker}(f_*)$ contains the class of a non-separating simple 
closed curve $\gamma$ we may reduce $\chi(M)$ by surgery on $\gamma$.
Combining the results of \cite{el81, em80, ga85}
we see that there is always such a curve $\gamma$.
Can this be shown directly, without appeal to \cite{el81, em80}?

We would like to study $PD_3$-complexes in a similar manner.
Let $X$ be a $PD_3$-complex and $f:M\to{X}$ a degree-1 map,
where $M$ is a closed 3-manifold.
Then $f$ is a homotopy equivalence $\Leftrightarrow$ $\mathrm{Ker}(f_*)=1$.
Since $\pi_1(M)$ and $\pi_1(X)$ are finitely presentable,
this kernel is normally generated by finitely many elements of $\pi_1(M)$, 
which may be represented by the components of a link $L\subset{M}$.
We would like to modify $M$ using such a link to render the kernel trivial.
This {\it is\/} possible if $X$ is homotopy equivalent 
to a closed orientable $3$-manifold $N$,
for $M$ may then be obtained from $N$ by Dehn surgery on a link 
whose components are null homotopic in $N$ \cite{gd07}.
Gadgil's argument appears to use the topology of the target space 
in an essential way.

The $PD_3$-complexes constructed in \S6 are not homotopy equivalent
to 3-manifolds,
so this strategy cannot be carried through in all cases.
However, it remains possible that every $PD_3$-complex 
is {\it virtually\/} a 3-manifold, i.e., has a finite covering space which is
homotopy equivalent to a closed orientable 3-manifold.
If this is true it must be posssible to kill $\mathrm{Ker}(f_*)$
by surgery and passing to finite covering spaces.

Easy reductions show that we may assume that $X$ is aspherical,
and then that the irreducible components of $M$ are aspherical.
There is then no need to pass to finite covers,
for if an aspherical $PD_3$-complex $X$ is virtually a $3$-manifold
then $X$ is homotopy equivalent to a $3$-manifold,
by the Geometrization Theorem of Thurston and Perelman,
and the work of Zimmermann \cite{zi82}.

Let $L=\amalg_{i\leq{m}}L_i$ be a link in a 3-manifold $M$ 
with an open regular neighbourhood $n(L)=\amalg_{i\leq{m}}n(L_i)$.
We shall say that $L$ admits a {\it drastic\/} surgery 
if there is a family of slopes
$\gamma_i\subset\partial{n(L_i)}$ such that the normal closure of
$\{[\gamma_1],\dots,[\gamma_n]\}$ in $\pi_1(M-n(L))$
meets the image of each peripheral subgroup $\pi_1(\partial{n(L_i)})$
in a subgroup of finite index.
If $f:M\to{N}$ is a degree-1 map of closed 3-manifolds $\mathrm{Ker}(f_*)$
is represented by a link which admits a drastic surgery \cite{gd07}.
(Gadgil's result is somewhat stronger.)

\begin{lemma} 
If $X$ is an aspherical $PD_3$-complex and $L$ admits a drastic surgery 
then $X$ is homotopy equivalent to a $3$-manifold.
\end{lemma}

\begin{proof}
After a drastic surgery on $L$ we may assume that $\mathrm{Ker}(f_*)$
is normally generated by finitely many elements of finite order.
Let $M=\#_{i=1}^{i=k}M_i$ be a factorization of $M$ 
as a connected sum of irreducible 3-manifolds, 
with $M_i$ aspherical if $i\leq{r}$ and $\pi_1(M_i)$ finite, 
$Z$ or $Z\oplus{Z/2Z}$ if $i>r$.
Since $X$ is aspherical $f$ extends to a map 
$F:\vee_{i=1}^{i=k}M_i\to{X}$.
If $\pi_1(M_i)$ is finite then $F|_{M_i}$ is null-homotopic,
while if $\pi_1(M_i)\cong{Z}$ or $Z\oplus{Z/2Z}$ then $F|_{M_i}$ 
factors through $S^1$.
In either case the restriction to such terms has degree 0.
Hence $F$ induces a degree-1 map from $g:N=\#_{i=1}^{i=r}M_i\to{X}$.
This map is clearly $\pi_1$-injective, and so it is a homotopy equivalence.
\end{proof}

There are knots which admit no drastic surgery.
The following example was suggested by Cameron Gordon.
Let $M$ be an orientable 3-manifold which is Seifert fibred over $S^2(p,q,r)$, 
where $\frac1p+\frac1q+\frac1r\leq1$, and let $K\subset{M}$ be a regular fibre.
Let $\phi,\mu\subset\partial{n(K)}$ be a regular fibre and a meridian, respectively.
Then surgery on the slope $s\mu+t\phi$ gives a 3-manifold which is Seifert fibred over $S^2(p,q,r,s)$,
if $s\not=0$, or is a connected sum of lens spaces, if $s=0$.
If $s\not=0$ the image of $\phi$ has infinite order in $\pi_1(N)$;
otherwise the image of $\mu$ has infinite order there.
Thus no surgery on a regular fibre of $M$ is drastic.
(We may modify this example to obtain one with $M$ not Seifert fibred, 
by replacing a tubular neighbourhood of another regular fibre 
by the exterior of a hyperbolic knot.)

However we have considerable latitude in our choice of link $L$ 
representing $\mathrm{Ker}(f_*)$.
In particular, we may modify $L$ by a link homotopy,
and so the key question may be:

\centerline{is every knot $K\subset{M}$ {\it homotopic\/} 
to one admitting a drastic surgery?}

\smallskip
\noindent The existence of $PD_3$-complexes which are not homotopy equivalent 
to 3-manifolds shows that we cannot expect a stronger result, in which
``contains $\dots\pi_1(\partial\overline{n(L_i)})$"
replaces ``meets the image $\dots$ finite index"
in the definition of drastic surgery.

In general, we might expect to encounter obstructions in $L_3(\pi,w)$
to obtaining a $\mathbb{Z}[\pi]$-homology equivalence by integral surgery.
For instance, there are finite groups of cohomological period 4 with 
finite Swan complexes but which do not act freely on homology 3-spheres 
\cite{hm86}.
The validity of the Novikov conjecture for aspherical 3-manifolds 
suggests that such obstructions may never arise in the cases of most interest 
to us.
(See \cite{jk03, kt02}.)
In any case, we allow more general Dehn surgeries.

The argument for the existence of a degree-1 map $f:M\to{X}$ 
does not require us to assume {\it a priori\/} that $X$ be finite, 
nor even that $\pi_1(X)$ be finitely presentable.
The latter condition is needed to ensure that $\mathrm{Ker}(f_*)$ is
represented by a link in $M$.
In all dimensions $n\geq4$ there are $PD_n$-groups of type $FF$
which are not finitely presentable \cite{da00}.
This leaves the question: are $PD_3$-groups finitely presentable?
Our strategy does not address this issue.

\end{document}